\documentclass[letterpaper, 10 pt, conference]{ieeeconf} 

\IEEEoverridecommandlockouts                              

\usepackage{hyperref}       
\hypersetup{colorlinks=true, linkcolor=blue, breaklinks=true, urlcolor=blue}

\usepackage{amsmath}
\usepackage{amssymb}

\usepackage{amsthm}
\usepackage[svgnames]{xcolor}
\usepackage{amsfonts}
\usepackage{centernot}
\usepackage{dsfont}
\usepackage{enumerate}
\usepackage{graphicx}
\usepackage{doi,xspace}
\usepackage{stmaryrd}
\usepackage{multirow}
\usepackage{cite}

\usepackage{version}

\includeversion{arxiv}\excludeversion{acc}\excludeversion{extra}

\newcommand*{\mydoi}[1]{\href{http://dx.doi.org/#1}{\includegraphics[width=.75em]{doi.png}}}

\newcommand{\KM}{Krasnosel’skii–Mann\xspace}

\newcommand{\suchthat}{\;\ifnum\currentgrouptype=16 \middle\fi|\;}
\newcommand{\until}[1]{\{1,\dots, #1\}}
\newcommand{\subscr}[2]{#1_{\textup{#2}}}

\newcommand{\setdef}[2]{\{#1 \; | \; #2\}}

\newcommand{\map}[3]{#1: #2 \rightarrow #3}

\newcommand{\real}{\mathbb{R}}
\newcommand{\realpositive}{\mathbb{R}_{>0}}

\newcommand{\realnonnegative}{\mathbb{R}_{\geq0}}

\newcommand{\complex}{\mathbb{C}}

\newcommand{\scirc}{\raise1pt\hbox{$\,\scriptstyle\circ\,$}}

\newcommand{\prox}{\mathrm{prox}}
\newcommand\oprocendsymbol{\hbox{$\square$}}
\newcommand\oprocend{\relax\ifmmode\else\unskip\hfill\fi\oprocendsymbol}

\definecolor{Gray}{gray}{0.9}

\newcommand{\bigO}{\mathcal{O}}

\newcommand{\sashatodo}[1]{\par\noindent{\raggedright\color{cyan}\texttt{From
			Sasha: #1}\par\marginpar{$\star$}}}

\DeclareSymbolFont{bbold}{U}{bbold}{m}{n}
\DeclareSymbolFontAlphabet{\mathbbold}{bbold}
\DeclareMathAlphabet{\mymathbb}{U}{BOONDOX-ds}{m}{n}

\newcommand{\vectorzeros}[1][]{\mymathbb{0}_{#1}}

\newcommand{\ds}{\displaystyle}

\newtheorem{theorem}{Theorem}
\newtheorem{proposition}[theorem]{Proposition}
\newtheorem{lemma}[theorem]{Lemma}

\newtheorem{definition}[theorem]{Definition}
\newtheorem{example}[theorem]{Example}

\newtheorem{remark}[theorem]{Remark}
\newtheorem{algo}[theorem]{Algorithm}

\graphicspath{{./images/}{./fig}}

\newcommand{\jac}[1]{D\mkern-0.75mu{#1}}

\newcommand{\seminorm}[1]{{\left\vert\kern-0.25ex\left\vert\kern-0.25ex\left\vert #1
		\right\vert\kern-0.25ex\right\vert\kern-0.25ex\right\vert}}

\newcommand{\semimeasure}[1]{\mu_{\seminorm{\cdot}}\kern-0.5ex\left(#1\right)}

\newcommand{\inprod}[2]{\langle{#1},{#2}\rangle}

\newcommand{\osL}{\operatorname{osL}}
\newcommand{\Lip}{\operatorname{Lip}}

\newcommand{\norm}[2]{\|#1\|_{#2}}

\newcommand{\zero}{\operatorname{Zero}}
\newcommand{\fixed}{\operatorname{Fix}}
\DeclareMathOperator*{\argmin}{arg\,min}

\DeclareMathOperator{\dom}{Dom}

\newcommand{\diagL}{\operatorname{diagL}}

\DeclareSymbolFont{bbold}{U}{bbold}{m}{n}
\DeclareSymbolFontAlphabet{\mathbbold}{bbold}

\renewcommand{\theenumi}{(\roman{enumi})}

\newcommand{\Id}{\mathsf{Id}}

\newcommand{\OF}{\mathsf{F}}
\newcommand{\OG}{\mathsf{G}}

\newcommand{\ON}{\mathsf{N}}
\newcommand{\OT}{\mathsf{T}}
\newcommand{\OC}{\mathsf{C}}
\newcommand{\OR}{\mathsf{R}}

\newcommand{\OJ}{\mathsf{J}}

\newcommand{\mcMV}{\mathcal{MV}}

\newcommand{\mcLD}{\mathcal{LD}}

\title{Non-Euclidean Monotone Operator Theory \\with Applications to Recurrent Neural Networks}
\author{Alexander Davydov$^{*}$, Saber Jafarpour$^{*}$, Anton
  V. Proskurnikov, and Francesco Bullo
    \thanks{$^*$ These authors contributed equally.}
  \thanks{This material is based upon work supported by the National Science Foundation Graduate Research Fellowship under Grant No.\ 2139319
    and AFOSR grant FA9550-22-1-0059.}
\thanks{Alexander Davydov and Francesco Bullo are with the Center for Control, Dynamical Systems,
	and Computation, University of California, Santa Barbara, 93106-5070, USA.
	{\tt\{davydov, bullo\}@ucsb.edu}. }
\thanks{Saber Jafarpour is with the School of Electrical and Computer Engineering, Georgia Institute of Technology.~\tt{saber@gatech.edu}.}
\thanks{Anton V. Proskurnikov is with the Department of Electronics and
	Telecommunications, Politecnico di Torino, Turin,
	Italy.~\tt{anton.p.1982@ieee.org}.}
}

\newcommand{\dmin}{\subscr{d}{min}}
\newcommand{\dmax}{\subscr{d}{max}}

\begin{document}
\maketitle

\begin{abstract}
  We provide a novel transcription of monotone operator theory to the
  non-Euclidean finite-dimensional spaces $\ell_1$ and $\ell_\infty$. We
  first establish properties of mappings which are monotone with respect
  to the non-Euclidean norms $\ell_1$ or $\ell_\infty$. In analogy with their
  Euclidean counterparts, mappings which are monotone with respect to a
  non-Euclidean norm are amenable to numerous algorithms for computing
  their zeros. We demonstrate that several classic iterative methods for
  computing zeros of monotone operators are directly applicable in the
  non-Euclidean framework. We present a case-study in the equilibrium
  computation of recurrent neural networks and demonstrate that casting the
  computation as a suitable operator splitting problem improves convergence
  rates.
\end{abstract}

\thispagestyle{empty}
\pagestyle{empty}

\section{Introduction}
In the last few years, monotone operator methods have become prevalent to
solve problems in optimization and control~\cite{AS:17,AB-ED-AS:19}, game
theory~\cite{LP:20}, systems analysis~\cite{TC-FF-RS:21}, and to better
understand machine learning models~\cite{EW-JZK:20,PLC-JCP:20b}. However,
monotone operator techniques are primarily based on the theory of Hilbert
and Euclidean spaces, while many problems are well-posed or better-suited
for analysis in a Banach space or finite-dimensional non-Euclidean
space. For example, in machine learning, it is known that robustness
analysis of artificial neural networks is naturally performed via the
$\ell_\infty$ norm and that such a norm is most appropriate for
high-dimensional input data such as images. Additionally, in the field of
robust control, $\mathcal{H}_\infty$ techniques are naturally stated over an
infinite-dimensional Banach space, so monotone operator techniques do not
apply.

\textit{Problem description and motivation:} In this paper, we aim to provide a natural transcription of many monotone operator techniques for computing zeros of monotone operators for operators which are naturally ``monotone" with respect 
to an $\ell_1$ or $\ell_\infty$ norm in a finite-dimensional space.

Monotone operator theory is a fertile field of nonlinear functional
analysis that generalizes the notion of monotone functions on $\real$ to
mappings on arbitrary Hilbert spaces and examines the properties of such maps. In
particular, an integral component of monotone operator theory is the design
of algorithms to compute zeros of monotone operators. This aspect makes
monotone operator theory compatible with convex optimization since the
subdifferential of any convex function is monotone and minimizing a convex
function is synonymous with finding a zero of its subdifferential. To this
end, there has been an extensive amount of work in the last decade in
applying monotone operator theory to convex optimization; e.g.,
see~\cite{EKR-SB:16,PLC:18,EKR-WY:21}.



Through the lens of duality theory, the theory of dissipative and accretive
operators on Banach spaces mirrors monotone operators on Hilbert
spaces to a degree~\cite{KL:85}. Despite these parallels, the theory of dissipative and
accretive operators has largely focused on iteratively computing solutions
of integral equations and PDEs in $L_p$ spaces for $p \neq 2$;
see~\cite{CC:09} for a relevant textbook. Moreover, many works in this
direction focus on Banach spaces that additionally have a uniformly smooth
or uniformly convex structure; this structure is not possesed by the
finite-dimensional $\ell_1$ and $\ell_\infty$ spaces. Ultimately, in
contrast to monotone operator theory over Hilbert spaces, the theory of
dissipative and accretive operators has found far fewer direct applications
to systems, control, and machine learning.

A notion similar to a monotone operator in a Hilbert space is that of a
contracting vector field~\cite{WL-JJES:98}. In fact, a vector field
$\map{\OF}{\real^n}{\real^n}$ is contracting with respect to an $\ell_2$
norm if and only if the negative vector field $-\OF$ is monotone when
thought of as an operator. However, vector fields are not restricted to
being contracting with respect to a Euclidean norm. In general, a vector
field may be contracting with respect to a non-Euclidean norm but not a
Euclidean one~\cite{ZA-EDS:14b}. Recently, there has been an increased
interest in studying vector fields that are contracting with respect to the
non-Euclidean norms $\ell_1$ and
$\ell_\infty$~\cite{ZA-EDS:14,SC:19,AD-SJ-FB:20o}. Due to the connection
between monotone operators and contracting vector fields, it is of interest
to explore the properties of operators that may be thought of as monotone
with respect to an $\ell_1$ or $\ell_\infty$ norm.

\textit{Contributions:} To facilitate the application of monotone
operator theory techniques to problems naturally arising in finite-dimensional
non-Euclidean spaces, we propose a novel non-Euclidean monotone operator
framework based on the theory of logarithmic norms~\cite{GS:06} (also known
as matrix measures). We use the logarithmic norm as a direct substitute for
inner-products in Hilbert spaces and we demonstrate that many classic
results from monotone operator theory directly carry over to their
non-Euclidean counterparts. Specifically, we show that the resolvent and
reflected resolvent operators of a non-Euclidean monotone operator have
properties analogous to those arising in Euclidean spaces.

Second, building upon the non-Euclidean monotone operator framework, we
demonstrate that classical iterative algorithms such as the forward step
method and proximal point method allow us to compute zeros of non-Euclidean monotone
operators in a manner identical to the procedure for traditional monotone operators. These results build upon both classical and modern works
on iterative methods for computing fixed points of non-expansive maps in
Banach spaces~\cite{SI:76,RC-JAS-JV:14}. We present estimates for Lipschitz
constants of these iterative methods and demonstrate that, for diagonally
weighted $\ell_1$ and $\ell_\infty$ norms, these algorithms achieve
improved rates of convergence compared to their Euclidean counterparts. As a clear distinction from the classical theory, we prove that the forward step method is convergent for an operator which is (weakly) monotone with respect to an $\ell_1$ or $\ell_\infty$ norm, but that the method need not converge if the operator is monotone with respect to a Euclidean norm. This result is analogous to the result on weakly-contracting ODEs as in~\cite[Theorem~21]{SJ-PCV-FB:19q}.

Third, we study operator splitting methods. We prove that the standard
forward-backward, Peaceman-Rachford, and Douglas-Rachford splitting
algorithms all apply in our framework and that improved convergence
may be achieved for these non-Euclidean norms compared to their Euclidean
counterparts.

Fourth, as an application, we present methods to compute equilibria for
recurrent neural networks. We extend the recent work
of~\cite{SJ-AD-AVP-FB:21f,AD-AVP-FB:21k} to demonstrate that our
non-Euclidean monotone operator theory is readily applicable and can
provide accelerated convergence of iterations when viewing the problem of
computing an equilibrium as an appropriate operator splitting problem. We
highlight several iterations for the computation of the equilibrium and
discuss the trade-off between computation, allowable range of stepsizes, and
rate of convergence between the iterations. Finally, we present numerical
simulations presenting rates of convergence of the different iterations
when applied to this problem.

\begin{acc}
In the interest of brevity, many proofs of technical lemmas are omitted and
will appear in a forthcoming technical report.
\end{acc}
\begin{arxiv}
	Since this document is an arXiv technical report, it contains proofs of additional technical lemmas that are not presented in the conference version.
\end{arxiv}

\section{Preliminaries}
\subsection{Notations}
For differentiable $\map{\OF}{\real^n}{\real^n}$, we let $\jac{\OF}(x) := \frac{\partial \OF(x)}{\partial x} \in \real^{n \times n}$ denote 
its Jacobian evaluated at $x$. For an arbitrary mapping $\OF$, we let $\dom(\OF)$ be its domain. For $\map{\OF}{\real^n}{\real^n}$, 
we let $\zero(\OF) := \setdef{x \in \real^n}{\OF(x) = 0}$ and $\fixed(\OF) = \setdef{x \in \real^n}{\OF(x) = x}$ be the sets of zeros of $\OF$ 
and fixed points of $\OF$, respectively. We let $\map{\Id}{\real^n}{\real^n}$ be the identity map and $I_n \in \real^{n \times n}$ be the $n \times n$ 
identity matrix. 

\subsection{Norms and Logarithmic Norms}
Instrumental to the theory of non-Euclidean monotone operator theory are logarithmic norms (also referred to as matrix measures), henceforth called log norms, independently discovered by Dahlquist and Lozinskii in 1958~\cite{GD:58,SML:58}. 
\begin{definition}[Logarithmic norm]
	Let $\|\cdot\|$ be a norm on $\real^n$ and its corresponding induced norm on $\real^{n \times n}$. The \emph{logarithmic norm} of a matrix $A \in \real^{n \times n}$ is
	\begin{equation}\label{eq:matrixmeasure}
	\mu(A) := \lim_{h \to 0^+} \frac{\|I_n + hA\| - 1}{h}.
	\end{equation}
\end{definition}
It is well known that this limit is well posed because the right-hand side of \eqref{eq:matrixmeasure}
is non-increasing in $h$, due to the convexity of the norm. We refer to
\cite{CAD-HH:72} for properties enjoyed by log norms, which include subadditivity, positive homogeneity, convexity, and $\alpha(A) \leq \mu(A) \leq \|A\|$. 

We will be specifically interested in diagonally weighted $\ell_1$ and $\ell_\infty$
norms defined by
\begin{gather*}
\norm{x}{1,[\eta]} = \sum_i\eta_i|x_i|
\qquad\text{and}\qquad
\norm{x}{\infty,[\eta]^{-1}} = \max_{i}  \frac{1}{\eta_i}|x_i|,
\end{gather*}
where, given a positive vector $\eta\in\realpositive^n$, we use $[\eta]$ to denote the diagonal matrix with diagonal entries $\eta$.  For $A \in \real^{n \times n}$, the corresponding induced and log norms are
\begin{align*}
\norm{A}{\infty,[\eta]^{-1}} &= \max_{i\in\until{n}} \sum_{j=1}^n \frac{\eta_j}{\eta_i} |a_{ij}|,
\\
\mu_{\infty,[\eta]^{-1}}(A) &= \max_{i \in \until{n}} \Big(a_{ii} + \sum\nolimits_{j=1,j \neq i}^n |a_{ij}|\frac{\eta_j}{\eta_i}\Big), \\
\|A\|_{1,[\eta]} &= \|A^\top\|_{\infty,[\eta]^{-1}}, \quad \mu_{1,[\eta]}(A) = \mu_{\infty,[\eta]^{-1}}(A^\top).
\end{align*}
We note also that for the Euclidean norm $\|\cdot\|_2$, the corresponding log norm is $\mu_2(A) = \frac{1}{2}\lambda_{\max}(A + A^\top)$.
\begin{extra}
\sashatodo{This is an interesting definition that would allow us to generalize our results to some classes of continuous $\OF$ (including locally Lipschitz $\OF$). After thinking about it, I think I will focus only on continuously differentiable $\OF$ and add a remark that results extend to locally Lipschitz $\OF$.}
We introduce the following useful definition:
\begin{definition}[Log norm-bounded mean-value set]
	Let $\|\cdot\|$ be a norm with corresponding log norm $\mu(\cdot)$ and let $\map{\OF}{\real^n}{\real^n}$.
	Then for $b \in \real$, define the effective log norm-bounded domain of $\OF$ by
	\begin{equation*}
	\begin{aligned}
	\mcLD_{\|\cdot\|}^{\leq b}(\OF) := \{&x \in \real^n\;|\;\forall y \in \real^n, \exists A \in \real^{n \times n} \text{ s.t. } \\&\OF(x) - \OF(y) = A(x-y) \text{ and } \mu(A) \leq b\}.
	\end{aligned}
	\end{equation*}
	If there exists $b \in \real$ such that $\mcLD_{\|\cdot\|}^{\leq b}(\OF) = \real^n$, we say that $\OF$ satisfies the \emph{log norm mean-value property with constant $b$}.
	If $\OF$ satisfies the log norm mean-value property with constant $b$, then define the log norm mean-value set by
	\begin{equation}
	\begin{aligned}
	\mcMV_{\|\cdot\|}^{\leq b}(\OF) := \{&A \in \real^{n \times n} \;|\; \exists x,y \in \real^n \text{ s.t. } \\&\OF(x) - \OF(y) = A(x-y) \text{ and } \mu(A) \leq b\}.
	\end{aligned}
	\end{equation}
\end{definition}
Note that if $\OF(x) = Ax + b$, then $\mcLD_{\|\cdot\|}^{\leq b}(\OF) = \real^n$ if and only if $\mu(A) \leq b$. In this case, $\mcMV_{\|\cdot\|}^{\leq b}(\OF) = \{A\}$.
Suppose $\map{\OF}{\real^n}{\real^n}$ is continuously differentiable and satisfies $\mu(\jac{\OF}(x)) \leq b$ for all $x \in \real^n$. Then $\mcLD_{\|\cdot\|}^{\leq b}(\OF) = \real^n$ and $\setdef{\int_{0}^1 \jac{\OF}(y - \tau(x-y))d\tau}{x,y \in \real^n} \subseteq \mcMV_{\|\cdot\|}^{\leq b}(\OF)$. 
\end{extra}
\subsection{Contractions, nonexpansive maps, Banach-Picard and \KM iterations}
For the remainder of the paper, we assume all mappings are continuously differentiable unless otherwise stated.

\begin{definition}[Lipschitz continuity]
	Let $\|\cdot\|$ be a norm and $\map{\OF}{\real^n}{\real^n}$ be a map. $\OF$ is 
	\emph{Lipschitz continuous} with constant $\Lip(\OF) \in \realnonnegative$ if for all $x_1,x_2 \in \real^n$
	\begin{equation}
	\norm{\OF(x_1)-\OF(x_2)}{}\leq \Lip(\OF) \norm{x_1-x_2}{}.
	\end{equation}
\end{definition}


Equivalently, $\OF$ is Lipschitz continuous with
constant $\Lip(\OF)$ if and only if
\begin{equation}
\norm{\jac{\OF}(x)}{} \leq \Lip(\OF) \qquad \text{for all } x\in\real^n.
\end{equation}

\begin{definition}[One-sided Lipschitz functions~{\cite[Definition~26]{AD-SJ-FB:20o}}]
	Given a norm $\norm{\cdot}{}$ with corresponding log norm
	$\mu(\cdot)$, a map $\map{\OF}{\real^n}{\real^n}$ is one-sided Lipschitz
	with constant $\osL(\OF)\in\real$ if
	\begin{equation} \label{eq:osL=WSIP-appendix}
	\mu(\jac{\OF}(x)) \leq \osL(\OF)   \qquad \text{for all } x\in\real^n.
	\end{equation}
\end{definition}


Note that (i) the one-sided
Lipschitz constant is upper bounded by the Lipschitz constant, (ii) a
Lipschitz map is always one-sided Lipschitz, and (iii) the one-sided
Lipschitz constant may be negative.

\begin{definition}[Contractions and nonexpansive maps]
	Let $\map{\OT}{\real^n}{\real^n}$ be Lipschitz with respect to a norm $\|\cdot\|$. We say
	\begin{enumerate}
		\item $\OT$ is a \emph{contraction} if $\Lip(\OT) \in {[0,1[}$,
		\item $\OT$ is \emph{nonexpansive} if $\Lip(\OT) = 1$.
	\end{enumerate}
\end{definition}


\begin{definition}[Averaged maps]
	We say a nonexpansive map $\map{\OT}{\real^n}{\real^n}$ is \emph{averaged} provided that there exists a nonexpansive map $\map{\ON}{\real^n}{\real^n}$ such that for some $\theta \in {]0,1[}$,
	\begin{equation}
	\OT = (1 - \theta) \Id + \theta \ON.
	\end{equation}
	
\end{definition}

\begin{remark}
	When the norm is induced by an inner product, the composition of two averaged mappings yields another averaged mapping; see~\cite[Proposition~4.44]{HHB-PLC:17}. This, however, need not hold for non-Euclidean spaces. 
\end{remark}

\begin{definition}[\KM iterations~{\cite[Section~5.2]{HHB-PLC:17}}]
	Let $\map{\OT}{\real^n}{\real^n}$ be nonexpansive with respect to a norm $\|\cdot\|$. The \emph{\KM iterations} applied to $\OT$ defines the sequence $\{x_k\}_{k=0}^\infty$ by
	\begin{equation}\label{eq:KMiterations}
	x_{k+1} = (1-\theta)x_{k} + \theta \OT(x_{k}),
	\end{equation}
	where $\theta \in {]0,1[}$.
\end{definition}

\begin{lemma}[Convergence and asymptotic regularity of \KM iterations~\cite{RC-JAS-JV:14}] \label{lemma:KMconvergence}
	Let $\map{\OT}{\real^n}{\real^n}$ be nonexpansive with respect to a norm $\|\cdot\|$ and consider the \KM iterations as in~\eqref{eq:KMiterations}. Suppose $\fixed(\OT) \neq \emptyset$ and let $x^* \in \fixed(\OT)$. Then
	\begin{equation}
	\|x_k - \OT(x_k)\| \leq \frac{2\|x_0 - x^*\|}{\sqrt{k\pi\theta(1-\theta)}}.
	\end{equation}
	In particular, $\|x_k - \OT(x_k)\| \to 0$ as $k \to \infty$ with $\|x_k - \OT(x_k)\| \sim \bigO(1/\sqrt{k}).$ Moreover, the convergence rate is optimized with $\theta = 1/2$. 
\end{lemma}
\section{Non-Euclidean Monotone Operators}
\subsection{Definitions and Properties}

\begin{definition}[Non-Euclidean monotone operator]\label{def:monotoneoperator}
	A continuously differentiable operator $\map{\OF}{\real^n}{\real^n}$ is strongly monotone with monotonicity parameter $c > 0$ with respect to a norm $\|\cdot\|$ on $\real^n$ provided for all $x \in \real^n$,
	\begin{equation}\label{eq:monotone}
	-\mu(-\jac{\OF}(x)) \geq c.
	\end{equation}
	If the inequality holds with $c = 0$, we say $\OF$ is monotone (or weakly monotone) with respect to $\|\cdot\|$. 
\end{definition}
Note that this condition is equivalent to $-\osL(-\OF) \geq c$. Moreover, if $\OF$ is only locally Lipschitz, we ask that~\eqref{eq:monotone} holds almost everywhere.
\begin{remark}[Comparison to the Euclidean case]
	For an operator $\map{\OF}{\real^n}{\real^n}$, let $\|\cdot\|_2$ be the Euclidean norm with corresponding inner product $\inprod{\cdot}{\cdot}$. 
Then following~\cite[Definition~20.1]{HHB-PLC:17}, $\OF$ is monotone with respect to $\|\cdot\|_2$ if
	$$\inprod{\OF(x) - \OF(y)}{x-y} \geq 0, \qquad \text{for all } x,y \in \real^n.$$
	If $\OF$ is continuously differentiable, this condition is known to be equivalent to (e.g.,~\cite{EKR-SB:16}) $\jac{\OF}(x) + \jac{\OF}(x)^\top \succeq 0$, or equivalently $-\mu_2(-\jac{\OF}(x)) \geq 0$ or $\frac{1}{2}\lambda_{\min}(\jac{\OF}(x) + \jac{\OF}(x)^\top) \geq 0$, which coincides with Definition~\ref{def:monotoneoperator}.
\end{remark}

By subadditivity of $\mu$, a sum of operators which are monotone with respect to the same norm is also monotone. Additionally, if $\OF$ is (strongly) 
monotone with monotonicity parameter $c \geq 0$, then for any $\alpha \geq 0$, $\Id + \alpha \OF$ is strongly monotone with monotonicity parameter $1 + \alpha c$. 
\begin{remark}[Connection with contracting vector fields~\cite{WL-JJES:98}]\label{rmk:contractiontheory}
	A continuously differentiable mapping $\map{\OF}{\real^n}{\real^n}$ is strongly contracting with rate $c > 0$ with respect to a 
norm $\|\cdot\|$ on $\real^n$ provided for all $x,y \in \real^n$,
	\begin{equation}
	\mu(\jac{\OF}(x)) \leq -c.
	\end{equation}
	If this inequality holds with $c = 0$, we say $\OF$ is weakly contracting with respect to $\|\cdot\|$. 
	Clearly, $\OF$ is (strongly) monotone if and only if $-\OF$ is (strongly) contracting.
\end{remark} 

\begin{example}
	An affine mapping $\OF(x) = Ax + b$ is monotone if and only if $-\mu(-A) \geq 0$ and strongly monotone with parameter $c$ if and only 
if $-\mu(-A) \geq c$. This implies that the spectrum of $A$ lies in the portion of the complex plane given by $\setdef{z \in \complex}{\Re(z) \geq c}$. 
\end{example}

\begin{arxiv}
\begin{lemma}[Lipschitz constants of inverses of strongly monotone operators]\label{lemma:inverse}
	Suppose $\OF: \real^n \to \real^n$ is a strongly monotone operator with parameter $c > 0$. Then $\OF^{-1}$ is Lipschitz with 
constant $\ell = 1/c$.
\end{lemma}
To prove Lemma~\ref{lemma:inverse}, we leverage the following useful property of log norms.
\begin{proposition}[Product property of log norms~\cite{CAD-HH:72}]\label{prop:productprop}
	Let $A \in \real^{n \times n}$ and $\|\cdot\|$ be a norm on $\real^n$ with corresponding log norm $\mu(\cdot)$. Then for all $x \in \real^n$,
	\begin{equation}
	\|Ax\| \geq \max\{-\mu(-A), -\mu(A)\} \|x\|.
	\end{equation}
\end{proposition}
\begin{proof}[Proof of Lemma~\ref{lemma:inverse}]
	Note that by the mean-value theorem and Proposition~\ref{prop:productprop},
	\begin{align*}
	\|\OF(x) - \OF(y)\| &= \left\|\int_{0}^1 \jac{\OF}(y + \tau(x-y))d\tau (x-y)\right\| \\
	&\geq -\mu\Big(-\int_{0}^1 \jac{\OF}(y + \tau(x-y))d\tau\Big)\|x - y\| \\
	&\geq \int_{0}^1 -\mu(-\jac{\OF}(y + \tau(x-y))d\tau)\|x - y\| \\
	& \geq c\|x- y\|.
	\end{align*}
	where second inequality is by subadditivity and continuity of $\mu$ and the final inequality is by the assumption of strong monotonicity. We can then immediately see that if $\OF(x) = \OF(y)$, then necessarily $x = y$, which implies 
that $\OF^{-1}$ is a mapping. Then write $x = \OF^{-1}(u), y = \OF^{-1}(v)$ (and therefore $\OF(x) = u, \OF(y) = v$). Then
	\begin{equation}
	\|u - v\| \geq c\|x - y\| = c\|\OF^{-1}(u) - \OF^{-1}(v)\|,
	\end{equation}
	which shows that $\OF^{-1}$ has Lipschitz constant $1/c$.
\end{proof}
\end{arxiv}

\begin{lemma}\label{lemma:fwdstepLip}
	Let $\map{\OF}{\real^n}{\real^n}$ be globally Lipschitz with respect to a diagonally-weighted $\ell_1$ or $\ell_\infty$ norm $\|\cdot\|$ with constant $\Lip(\OF) = \ell$.
	If $\OF$ is (possibly strongly) monotone with respect to $\|\cdot\|$ with monotonicity parameter $c \geq 0$, then 
		\begin{equation}
		\Lip(\Id - \alpha\OF) = 1 - \alpha c, \qquad \text{for all } \alpha \in {\Big]0, \frac{1}{\diagL(\OF)}\Big]},
		\end{equation}
		where $\diagL(\OF) := \sup_{x \in \real^n}\max_{i\in\until{n}} (\jac{\OF}(x))_{ii} \leq \ell$. 
\end{lemma}
\begin{arxiv}
\begin{proof}
	The result follows from~\cite[Theorem~2]{SJ-AD-AVP-FB:21f}. 
\end{proof}
\end{arxiv}

	Note that for Euclidean norms, if $\OF$ is monotone, but not strongly monotone, then $(\Id - \alpha\OF)$ need not be nonexpansive for any $\alpha>0$. Indeed, consider $\OF(x) = \left(\begin{smallmatrix} 0 & 1 \\ -1 & 0 \end{smallmatrix}\right)x$, which is monotone with respect to the $\ell_2$ norm, but $(\Id - \alpha\OF)$ is expansive for every $\alpha > 0$.

\subsection{Resolvent and reflected resolvent operators}
\begin{definition}[Resolvent and reflected resolvent]
	Let $\map{\OF}{\real^n}{\real^n}$ be a mapping and $\alpha > 0$. The resolvent of $\alpha\OF$ is defined as
	\begin{equation}
	\OJ_{\alpha\OF} = (\Id + \alpha\OF)^{-1}.
	\end{equation}
	The reflected resolvent, also called the Cayley operator of $\alpha\OF$ is 
	\begin{equation}
	\OR_{\alpha\OF} = 2 \OJ_{\alpha\OF} - \Id.
	\end{equation}
\end{definition}
Note that for any $\alpha > 0$, we have $\OF(x) = 0$ if and only if $x = \OJ_{\alpha\OF}(x) = \OR_{\alpha\OF}(x)$.

\begin{theorem}[A non-Euclidean Minty-Browder theorem]
	Suppose $\map{\OF}{\real^n}{\real^n}$ is monotone. Then for every $\alpha > 0$, $\dom(\OJ_{\alpha\OF}) = \dom(\OR_{\alpha\OF}) = \real^n$.
\end{theorem}
\begin{arxiv}
\begin{proof}
	Note that 
	$$(\Id + \alpha \OF)(x) = 0 \quad \iff \quad -\alpha \OF(x) = x.$$
	However, since $\OF$ is continuously differentiable and $\OF$ is monotone, 
	$$-\mu(-\jac{\OF}(x)) \geq 0, \text{ for all } x \; \implies \; \mu(-\alpha \jac{\OF}(x)) < 1,$$
	for all $x \in \real^n$. Then by~\cite[Theorem~1]{SJ-AD-AVP-FB:21f}, $-\alpha \OF(x) = x$ has a unique solution, so $(\Id + \alpha \OF)(x) = 0$ has a unique solution. Moreover, for every $u \in \real^n$, the 
mapping $x \mapsto \alpha \OF(x) - u$ is continuously differentiable and monotone and thus has a unique fixed point, implying for every $u \in \real^n$, there exists an $x \in \real^n$ such that $(\Id + \alpha \OF)(x) = u$. This proves that $\dom(\OJ_{\alpha\OF}) = \real^n$. The proof for the reflected resolvent is a straightforward consequence of $\dom(\OJ_{\alpha\OF}) = \real^n$.
\end{proof}
\end{arxiv}

\begin{lemma}[Lipschitz constant of the resolvent operator]\label{lemma:contractionResolvent}
	Suppose $\OF: \real^n \to \real^n$ is (strongly) monotone with parameter $c
	\geq 0$. Then for every $\alpha > 0$, 
	\begin{equation}
	\Lip(\OJ_{\alpha\OF}) = \frac{1}{1 + \alpha c}.
	\end{equation}
\end{lemma}
\begin{arxiv}
\begin{proof}
	We observe that $\Id + \alpha\OF$ is
	strongly monotone with parameter $1 + \alpha c$. Then by
	Lemma~\ref{lemma:inverse}, the result holds.
\end{proof}
\end{arxiv}

\begin{lemma}[Reflected resolvent characterization~\cite{EKR-SB:16}]\label{lemma:Cayley}
	Suppose $\map{\OF}{\real^n}{\real^n}$ is monotone and $\alpha \geq 0$. Then
	$$\OR_{\alpha\OF} = (\Id - \alpha \OF)(\Id + \alpha \OF)^{-1}.$$
\end{lemma}
\begin{theorem}[Lipschitz constant of the Cayley operator] \label{thm:CayleyLip}
	Suppose $\OF: \real^n \to \real^n$ is globally Lipschitz with constant $\ell$ with respect to a diagonally weighted $\ell_1$ or $\ell_\infty$ norm. Moreover, suppose $\OF$ is (strongly) monotone with respect to $\|\cdot\|$ with monotonicity parameter $c
		\geq 0$. Then for $\alpha \in {]0,
			\frac{1}{\diagL(\OF)}[}$, 
		\begin{equation}
		\Lip(\OR_{\alpha\OF}) = \frac{1 - \alpha c}{1 + \alpha c} \leq 1.
		\end{equation}
\end{theorem}
\begin{proof}
	By Lemma~\ref{lemma:fwdstepLip},
	$\Lip(\Id - \alpha \OF) = 1 - \alpha c$ for $\alpha \in {]0, \frac{1}{\diagL(\OF)}]}$. Therefore, the result follows from Lemmas~\ref{lemma:contractionResolvent} and~\ref{lemma:Cayley} since the Lipschitz constant 
of a composition of Lipschitz maps is the product of the Lipschitz constants.
\end{proof}

\begin{arxiv}
\begin{lemma}[Averagedness of resolvent]\label{lemma:averagedResolvent}
	Suppose $\OF$ is Lipschitz and monotone with respect to a diagonally weighted $\ell_1$ or $\ell_\infty$ norm. Then for every $\alpha > 0$, $\OJ_{\alpha\OF}$ is averaged.
\end{lemma}
\begin{proof}
	Consider the auxiliary operator $$\OC_\OF^\theta := \frac{\OJ_{\alpha\OF}}{\theta} - \frac{1 - \theta}{\theta} \Id,$$
	for $\theta \in {]0,1[}$. Note that the reflected resolvent corresponds to $\theta = \frac{1}{2}$. Then it is straightforward to compute
	\begin{align*}
	\OC_\OF^\theta &= \frac{(\Id + \alpha \OF)^{-1}}{\theta} - \frac{1 - \theta}{\theta}(\Id + \alpha \OF)(\Id + \alpha \OF)^{-1} \\ 
	&= \Big(\frac{\Id}{\theta} - \frac{1 - \theta}{\theta}(\Id + \alpha\OF)\Big)(\Id + \alpha \OF)^{-1} = \Big(\Id - \frac{(1 - \theta)\alpha}{\theta}\OF\Big)\OJ_{\alpha\OF}.
	\end{align*}
	Since $\OF$ is monotone, $\OJ_{\alpha\OF}$ is nonexpansive, and by Lemma~\ref{lemma:fwdstepLip}, 
	$$\Lip\Big(\Id - \frac{(1-\theta)\alpha}{\theta}\OF\Big) = 1, \qquad \text{for all } \alpha \in {\Big]0, \frac{1 - \theta}{\theta\diagL(\OF)}\Big]},$$
	which implies that $\OC_\OF^\theta$ is nonexpansive for all $\alpha$ in this range.
	
	Let $\alpha > 0$ be arbitrary. Then for any 
	$$\theta \leq \frac{1}{1 + \alpha\diagL(\OF)} \in {]0,1[},$$
	we have that 
	$$\OJ_{\alpha\OF} = (1 - \theta)\Id + \theta \OC_\OF^\theta$$
	and $\OC_\OF^\theta$ is nonexpansive. This proves that $\OJ_{\alpha\OF}$ is averaged.
\end{proof}
\end{arxiv}
\begin{example}
	Consider the linear operator
	$$\OF(x) = Ax = \begin{pmatrix} 2 & -2 \\ 1 & 1 \end{pmatrix}x.$$
	Then clearly $\OF$ is monotone with respect to the $\ell_\infty$ norm since
	$\ds-\mu_\infty(-A) = -\mu_\infty\left(\begin{smallmatrix} -2 & 2 \\ -1 & -1 \end{smallmatrix}\right) = 0.$
	Then for $\alpha = 1$, we compute
	$$\OJ_{\alpha\OF}(x) = \begin{pmatrix} 1/4 & 1/4 \\ -1/8 & 3/8 \end{pmatrix} x, \quad \OR_{\alpha\OF}(x) = \begin{pmatrix} -1/2 & 1/2 \\ -1/4 & -1/4 \end{pmatrix} x.$$
	Thus, $\Lip(\OJ_{\alpha\OF}) = 1/2$ and $\Lip(\OR_{\alpha\OF}) = 1$. In other words, for $\alpha = 1$, $\OJ_{\alpha\OF}$ is a contraction and $\OR_{\alpha\OF}$ is nonexpansive. \\
	For $\alpha = 2$, we compute
	$$\OJ_{\alpha\OF}(x) = \begin{pmatrix} \frac{3}{23} & \frac{4}{23} \\ -\frac{2}{23} & \frac{5}{23} \end{pmatrix} x, \quad \OR_{\alpha\OF}(x) = \begin{pmatrix} -\frac{17}{23} & \frac{8}{23} \\ -\frac{4}{23} & -\frac{13}{23} \end{pmatrix} x.$$
	Thus, $\Lip(\OJ_{\alpha\OF}) = 7/23$ and $\Lip(\OR_{\alpha\OF}) = 25/23$. In other words, for $\alpha = 2$, $\OJ_{\alpha\OF}$ is a contraction and $\OR_{\alpha\OF}$ is expansive.
\end{example}

\section{Finding Zeros of Non-Euclidean Monotone Operators}

\begin{table*}[]
	\centering
	\resizebox{.75\textwidth}{!}{	\begin{tabular}{|c|c|c|c|c|}
			\hline
			\multirow{3}{*}{Algorithm} & \multicolumn{4}{c|}{$\OF$ strongly monotone and globally Lipschitz} \\ \cline{2-5}
			& \multicolumn{2}{c|}{$\ell_2$} & \multicolumn{2}{c|}{Diagonally weighted $\ell_1$ or $\ell_\infty$} \\ \cline{2-5}
			& $\alpha$ range & Optimal $\Lip$ & $\alpha$ range & Optimal $\Lip$ \\ \hline 
			& & & & \\[-0.3em] 
			Forward step & $\ds{\Big]0, \frac{2c}{\ell^2}\Big[}$ & $\ds1-\frac{1}{2\kappa^2} + \bigO\Big(\frac{1}{\kappa^3}\Big)$ & $\ds{\Big]0, \frac{1}{\diagL(\OF)}\Big]}$ & $\ds1-\frac{1}{\kappa_\infty}$ \\ 
			& & & & \\[-0.3em]
			Proximal point & ${]0, \infty[}$ & N/A & ${]0, \infty[}$ & N/A \\ 
			& & & & \\[-0.3em]
			Cayley method & ${]0, \infty[}$ & $\ds1-\frac{1}{2\kappa}+\bigO\Big(\frac{1}{\kappa^2}\Big)$ & $\ds{\Big]0, \frac{1}{\diagL(\OF)}\Big]}$ & $\ds1-\frac{2}{\kappa_\infty}+ \bigO\Big(\frac{1}{\kappa_\infty^2}\Big)$ \\ \hline
	\end{tabular}}
	\caption{Step size ranges and Lipschitz constants for algorithms for finding zeros of monotone operators. For $\OF$ strongly monotone, let $c$ be its monotonicity parameter (with respect to the appropriate norm), $\ell$ its appropriate Lipschitz constant, and $\diagL(\OF) := \sup_{x \in \real^n}\max_{i\in\until{n}} (\jac{\OF}(x))_{ii} \leq \ell$. Additionally, $\kappa := \ell/c \geq 1$ and $\kappa_\infty := \diagL(\OF)/c \in [1,\kappa]$. Ranges of $\alpha$ and optimal Lipschitz constants for the Euclidean case are provided in~\cite{EKR-SB:16}. We do not assume that the strongly monotone $\OF$ is the gradient of a strongly convex function.}\label{table:comparison}
\end{table*}

Consider the problem of finding an $x \in \real^n$ that satisfies
\begin{equation}
\OF(x) = 0,
\end{equation}
where $\OF$ is monotone. This problem shows up in the computation of equilibrium points of contracting vector fields as noted in Remark~\ref{rmk:contractiontheory}. We present several well-known algorithms for finding zeros of monotone operators (see, e.g.,~\cite{EKR-SB:16}) and show how the non-Euclidean monotone operator framework allows the same algorithms to compute zeros of non-Euclidean monotone operators.

\begin{algo}[Forward step method]
	The \emph{forward step method} corresponds to the fixed point iteration
	\begin{equation}\label{eq:forwardstep}
	x_{k+1} = (\Id - \alpha\OF)(x_k).
	\end{equation}
\end{algo}

\begin{theorem}[Convergence guarantees for the forward step method]\label{thm:fwdstep}
	Suppose $\map{\OF}{\real^n}{\real^n}$ is globally Lipschitz with constant $\ell$ with respect to a diagonally-weighted $\ell_1$ or $\ell_\infty$ norm $\|\cdot\|$ and
	\begin{enumerate}
		\item\label{fwdstep1inf} $\OF$ is strongly monotone with respect to $\|\cdot\|$ with monotonicity parameter $c > 0$. Then the iteration~\eqref{eq:forwardstep} converges to the unique zero, $x^*$, of $\OF$ for every $\alpha \in {]0,\frac{1}{\diagL(\OF)}]}$. Moreover, for every $k \in \mathbb{Z}_{\geq 0}$, the iteration satisfies
		$$\|x_{k+1} - x^*\| \leq (1 - \alpha c)\|x_{k}-x^*\|,$$
		with the convergence rate being optimized at $\alpha = 1/\diagL(\OF)$. 
		\item\label{fwdstepmonotone} $\OF$ is monotone with respect to $\|\cdot\|$. Then if $\zero(\OF) \neq \emptyset$, the iteration~\eqref{eq:forwardstep} converges to an element of $\zero(\OF)$ for every $\alpha \in {]0,\frac{1}{\diagL(\OF)}[}.$ 
	\end{enumerate}
\end{theorem}
\begin{arxiv}
\begin{proof}
	Statement~\ref{fwdstep1inf} follows from Lemma~\ref{lemma:fwdstepLip}. Regarding Statement~\ref{fwdstepmonotone}, since $\OF$ is monotone with respect to a diagonally weighted $\ell_1$ or $\ell_\infty$ norm, $(\Id - \alpha \OF)$ is nonexpansive for $\alpha \in {]0,\frac{1}{\diagL(\OF)}]}$ by Lemma~\ref{lemma:fwdstepLip}. Moreover, for every $\alpha \in {]0,\frac{1}{\diagL(\OF)}[}$, there exists $\theta \in {]0,1[}$ such that
	$$\Id - \alpha\OF = (1-\theta) \Id + \theta(\Id - \tilde{\alpha}\OF),$$
	for some $\tilde{\alpha} \in {]0,\frac{1}{\diagL(\OF)}]}$. Therefore $\Id - \alpha \OF$ is averaged and by Lemma~\ref{lemma:KMconvergence}, if $\zero(\OF) \neq \emptyset$, the forward step method converges to an element of $\zero(\OF)$.
\end{proof}
\end{arxiv}
\begin{algo}[Proximal point method]\label{alg:proximalpt}
	The \emph{proximal point method} corresponds to the fixed point iteration
	\begin{equation}\label{eq:resolventIteration}
	x_{k+1} = \OJ_{\alpha\OF}(x_k) = (\Id + \alpha\OF)^{-1}(x_k).
	\end{equation} 
\end{algo}

\begin{theorem}[Convergence guarantees for the proximal point method]\label{thm:proximalconvergence}
	Suppose $\map{\OF}{\real^n}{\real^n}$ is
	\begin{enumerate}
		\item\label{proximalStrong} strongly monotone with respect to a norm $\|\cdot\|$ with monotonicity parameter $c > 0$. Then the iteration~\eqref{eq:resolventIteration} converges to the unique zero, $x^*$, of $\OF$ for every $\alpha \in {]0, \infty[}$. Moreover, for every $k \in \mathbb{Z}_{\geq 0}$, the iteration satisfies
		$$\|x_{k+1} - x^*\| \leq \frac{1}{1+\alpha c}\|x_{k}-x^*\|.$$
		\item\label{proximalLipschitz} monotone and globally Lipschitz with respect to a diagonally weighted $\ell_1$ or $\ell_\infty$ norm. Then if $\zero(\OF) \neq \emptyset$, the iteration~\eqref{eq:resolventIteration} converges to an element of $\zero(\OF)$ for every $\alpha \in {]0,\infty[}$.
	\end{enumerate}
\end{theorem}
\begin{proof}
	Statement~\ref{proximalStrong} holds due to Lemma~\ref{lemma:contractionResolvent}. Statement~\ref{proximalLipschitz} holds since Lipschitzness of $\OF$ implies that $\OJ_{\alpha\OF}$ is averaged together with Lemma~\ref{lemma:KMconvergence}.
\end{proof}


\begin{algo}
	The \emph{Cayley method} corresponds to the fixed point iteration
	\begin{equation}\label{eq:CayleyIteration}
	x_{k+1} = \OR_{\alpha\OF}(x_k) = 2(\Id + \alpha\OF)^{-1}(x_k) - x_k.
	\end{equation} 
\end{algo}	

\begin{theorem}[Convergence guarantees for the Cayley method]\label{thm:Cayley}
	Suppose $\map{\OF}{\real^n}{\real^n}$ is globally Lipschitz with constant $\ell$ with respect to a diagonally-weighted $\ell_1$ or $\ell_\infty$ norm $\|\cdot\|$ and
	\begin{enumerate}
		\item\label{Cayley1inf} $\OF$ is strongly monotone with respect to $\|\cdot\|$ with monotonicity parameter $c > 0$. Then the iteration~\eqref{eq:CayleyIteration} converges to the unique zero, $x^*$, of $\OF$ for every $\alpha \in {]0,\frac{1}{\diagL(\OF)}]}$. Moreover, for every $k \in \mathbb{Z}_{\geq 0}$, the iteration satisfies
		$$\|x_{k+1} - x^*\| \leq \frac{1 - \alpha c}{1 + \alpha c}\|x_{k}-x^*\|,$$
		with the convergence rate being optimized at $\alpha = 1/\diagL(\OF)$. 
		\item\label{CayleyAveraged} $\OF$ is monotone with respect to $\|\cdot\|$. Then if $\zero(\OF) \neq \emptyset$, the averaged iterations
		$$x_{k+1} = \frac{1}{2}x_k + \frac{1}{2}\OR_{\alpha\OF}(x_k)$$
		correspond to the proximal point iterations~\eqref{eq:resolventIteration}, which are guaranteed to converge to an element of $\zero(\OF)$ for every $\alpha \in {]0,\infty[}$. 
	\end{enumerate}
\end{theorem}
\begin{proof}
	Statement~\ref{Cayley1inf} follows from Theorem~\ref{thm:CayleyLip}. Statement~\ref{CayleyAveraged} holds since
	$\frac{1}{2}\Id + \frac{1}{2}(2\OJ_{\alpha\OF} - \Id) = \OJ_{\alpha\OF},$
	and convergence follows by Theorem~\ref{thm:proximalconvergence}\ref{proximalLipschitz} since $\zero(\OF) \neq \emptyset$.
\end{proof}

We provide a comparison of the range of step sizes and Lipschitz constants as provided by the classical monotone operator theory~\cite{EKR-SB:16} and Theorems~\ref{thm:fwdstep},~\ref{thm:proximalconvergence}, and~\ref{thm:Cayley} in Table~\ref{table:comparison}. Note that in Table~\ref{table:comparison} we do not assume that the strongly monotone $\OF$ is the gradient of a strongly convex function.

\section{Finding Zeros of a Sum of Non-Euclidean Monotone Operators}
In many instances, one may wish to execute the proximal point method, Algorithm~\ref{alg:proximalpt}, to compute a zero of a monotone operator $\map{\ON}{\real^n}{\real^n}$. However, in general, the implementation of the iteration~\eqref{eq:resolventIteration} may be hindered by the difficulty in evaluating $\OJ_{\alpha \ON}$. To remedy this issue, it is often assumed that $\ON$ can be expressed as the sum of two monotone operators $\OF$ and $\OG$ where the resolvent $\OJ_{\alpha\OG}$ may be easy to compute and $\OF$ satisfies some regularity condition. Alternatively, in some situations, decomposing $\ON = \OF+ \OG$ and finding $x \in \real^n$ such that $(\OF + \OG)(x) = 0$ provides additional flexibility in choice of algorithm and may improve convergence rates. 

Motivated by the above, we consider the problem of finding an $x \in \real^n$ such that
\begin{equation}
(\OF+\OG)(x) = 0,
\end{equation}
where $\map{\OF,\OG}{\real^n}{\real^n}$ are monotone with respect to a diagonally weighted $\ell_1$ or $\ell_\infty$ norm. 

\begin{algo}[Forward-backward splitting]\label{alg:FwdBwd}
	Assume $\alpha > 0$. Then by~\cite[Section~7.1]{EKR-SB:16}
	\begin{align*}
	(\OF + \OG)(x) = 0 \quad \iff x = \OJ_{\alpha\OG}(\Id - \alpha \OF)(x).
	\end{align*}
	The \emph{forward-backward splitting method} corresponds to the fixed point iteration
	\begin{equation}\label{eq:fwdbackward}
	x_{k+1} = \OJ_{\alpha\OG}(\Id - \alpha \OF)(x_k).
	\end{equation}
	Additionally, if both $\OF$ and $\OG$ are monotone, define the \emph{averaged forward-backward splitting iterations}
	\begin{equation}\label{eq:averagedfwdbackward}
	x_{k+1} = \frac{1}{2}x_k + \frac{1}{2}\OJ_{\alpha\OG}(\Id - \alpha \OF)(x_k).
	\end{equation}
\end{algo}
\begin{theorem}[Convergence guarantees for forward-backward splitting method]
	Suppose $\map{\OF}{\real^n}{\real^n}$ is globally Lipschitz with respect to a diagonally weighted $\ell_1$ or $\ell_\infty$ norm $\|\cdot\|$ and $\map{\OG}{\real^n}{\real^n}$ is monotone with respect to the same norm.
	\begin{enumerate}
		\item\label{fwdbackwardOneMonotone} If $\OF$ is strongly monotone with respect to $\|\cdot\|$ with monotonicity parameter $c > 0$, then the iteration~\eqref{eq:fwdbackward} converges to the unique zero, $x^*$, of $\OF+\OG$ for every $\alpha \in {]0, \frac{1}{\diagL(\OF)}]}$. Moreover, for every $k \in \mathbb{Z}_{\geq 0}$, the iteration satisfies
		$$\|x_{k+1} - x^*\| \leq (1 - \alpha c)\|x_{k}-x^*\|,$$
		with the convergence rate being optimized at $\alpha = 1/\diagL(\OF)$.
		\item\label{fwdbackwardBothMonotone} If $\OF$ is monotone with respect to $\|\cdot\|$ and $\zero(\OF + \OG) \neq \emptyset$, then the iteration~\eqref{eq:averagedfwdbackward} converges to an element of $\zero(\OF + \OG)$ for every $\alpha \in {]0, \frac{1}{\diagL(\OF)}]}$.
	\end{enumerate}
\end{theorem}
\begin{proof}
	Statement~\ref{fwdbackwardOneMonotone} follows from the fact that the Lipschitz constant of a composition of 
maps is the product of the Lipschitz constants together with Lemma~\ref{lemma:fwdstepLip}. Statement~\ref{fwdbackwardBothMonotone} follows from Lemma~\ref{lemma:KMconvergence}.
\end{proof}

Compared to the Euclidean case, if both $\OF$ and $\OG$ are monotone, then the averaged iterations~\eqref{eq:averagedfwdbackward} must be applied to compute a zero of $\OF + \OG$. In the Euclidean case, both $\OJ_{\alpha\OG}$ and $(\Id - \alpha \OF)$ are averaged and therefore the composition is also averaged. 


\begin{algo}[Peaceman-Rachford and Douglas-Rachford splitting]
	Let $\alpha > 0$. Then by~\cite[Section~7.3]{EKR-SB:16},
	\begin{equation}\label{eq:PRsplitting}
	(\OF + \OG)(x) = 0 \quad \iff \quad \OR_{\alpha\OF} \OR_{\alpha\OG} z = z \text{ and } x = \OJ_{\alpha\OG} z.
	\end{equation}
	The \emph{Peaceman-Rachford splitting method} corresponds to the fixed point iteration
	\begin{equation}\label{eq:PRiteration}
	\begin{aligned}
	x_{k+1/2} &= \OJ_{\alpha\OG}(z_k), \\
	z_{k+1/2} &= 2x_{k+1/2} - z_k, \\
	x_{k+1} &= \OJ_{\alpha\OF}(z_{k+1/2}), \\
	z_{k+1} &= 2x_{k+1} - z_{k+1/2}.
	\end{aligned}
	\end{equation}
	If both $\OF$ and $\OG$ are monotone, the term $\OR_{\alpha\OF} \OR_{\alpha\OG}$ in~\eqref{eq:PRsplitting} is averaged to yield the fixed point equation
	\begin{equation}\label{eq:DRsplitting}
	(\OF + \OG)(x) = 0 \; \iff \; \frac{1}{2}(\Id + \OR_{\alpha\OF} \OR_{\alpha\OG})z = z \text{ and } x = \OJ_{\alpha\OG} z.
	\end{equation}
	The fixed point iteration corresponding to~\eqref{eq:DRsplitting} is called the \emph{Douglas-Rachford splitting method} and is given by
	\begin{equation}\label{eq:DRiteration}
	\begin{aligned}
	x_{k+1/2} &= \OJ_{\alpha\OG}(z_k), \\
	z_{k+1/2} &= 2x_{k+1/2} - z_k, \\
	x_{k+1} &= \OJ_{\alpha\OF}(z_{k+1/2}), \\
	z_{k+1} &= z^k + x_{k+1} - x_{k+1/2}.
	\end{aligned}
	\end{equation}
\end{algo}

\begin{theorem}[Convergence guarantees for Peaceman-Rachford and Douglas-Rachford splitting methods]
	Suppose both $\map{\OF}{\real^n}{\real^n}$ and $\map{\OG}{\real^n}{\real^n}$ are globally Lipschitz with respect to a diagonally weighted $\ell_1$ or $\ell_\infty$ norm $\|\cdot\|$ and (without loss of generality) $\OG$ is monotone with respect to the same norm.
	\begin{enumerate}
		\item\label{PR} If $\OF$ is strongly monotone with respect to $\|\cdot\|$ with monotonicity parameter $c > 0$, then the sequence of $\{x_k\}_{k=0}^\infty$ generated by the iteration~\eqref{eq:PRiteration} converges to the unique zero, $x^*$, of $\OF+\OG$ for every $\ds\alpha \in {\Big]0, \min\Big\{\frac{1}{\diagL(\OF)},\frac{1}{\diagL(\OG)}\Big\}\Big]}$. Moreover, for every $k \in \mathbb{Z}_{\geq 0}$, the iteration satisfies
		$$\|x_{k+1} - x^*\| \leq \frac{1-\alpha c}{1 + \alpha c}\|x_{k}-x^*\|,$$
		with the convergence rate being optimized at $\alpha = \min\Big\{\frac{1}{\diagL(\OF)},\frac{1}{\diagL(\OG)}\Big\}$.
		\item\label{DR} If $\OF$ is monotone with respect to $\|\cdot\|$ and $\zero(\OF + \OG) \neq \emptyset$, then the sequence of $\{x_k\}_{k=0}^\infty$ generated by the iteration~\eqref{eq:DRiteration} converges to an element of $\zero(\OF + \OG)$ for every $\ds\alpha \in {\Big]0, \min\Big\{\frac{1}{\diagL(\OF)},\frac{1}{\diagL(\OG)}\Big\}\Big]}$.
	\end{enumerate}
\end{theorem}

\begin{proof}
	Statement~\ref{PR} holds by Theorem~\ref{thm:CayleyLip}, while statement~\ref{DR} holds by Theorem~\ref{thm:CayleyLip} and Lemma~\ref{lemma:KMconvergence}.
\end{proof}

\section{Application to recurrent neural networks}
\subsection{Analysis and various iterations}
Consider the continuous-time recurrent neural network
\begin{equation}\label{eq:inn}
\begin{aligned}
\dot{x} = -x + \Phi(Ax + Bu + b) =: \OF(x,u),\\
\end{aligned}
\end{equation}
where $x \in \real^n, u \in \real^m, A \in \real^{n \times n}, B \in \real^{n \times m}, b \in \real^n$, and $\map{\Phi}{\real^n}{\real^n}$ is an activation function applied entrywise, i.e., $\Phi(x) = (\phi(x_1),\dots,\phi(x_n))^\top$. In this example, we consider the case that $\phi$ is a LeakyReLU activation function, i.e., $\phi(x) = \max\{x,ax\}$ for some $a \in {]0,1[}$. In~\cite{AD-AVP-FB:21k}, it was shown that a sufficient condition for the contractivity of this neural network is the existence of weights $\eta \in \realpositive^n$ such that $\mu_{\infty,[\eta]^{-1}}(A) < 1$. If this condition holds, then the recurrent neural network~\eqref{eq:inn} is contracting with respect to $\|\cdot\|_{\infty,[\eta]^{-1}}$ with rate $1 - \phi(\mu_{\infty,[\eta]^{-1}}(A))$. In what follows, we define $\gamma := \mu_{\infty,[\eta]^{-1}}(A) < 1$.

Suppose that, for fixed $u$, we are interested in efficiently computing the unique equilibrium point $x^*(u)$ of $F(x,u)$. 
Since $\OF(x,u)$ is contracting with respect to $\|\cdot\|_{\infty,[\eta]^{-1}}$, $-\OF(x,u)$ is strongly monotone with monotonicity parameter $1 - \phi(\gamma)$. As a consequence, applying the forward step method, Algorithm~\ref{eq:forwardstep} to compute $x^*(u)$ yields the iteration
\begin{equation}\label{eq:averagediteration}
x_{k+1} = (1 - \alpha)x_k + \alpha \phi(Ax_k + Bu + b),
\end{equation}
which is the iteration proposed in~\cite{SJ-AD-AVP-FB:21f}. This iteration is guaranteed to converge for every $\ds\alpha \in {\Big]0,\frac{1}{1-\min_{i\in \until{n}}\min\{a \cdot (A)_{ii},(A)_{ii}\}}\Big]}$ with contraction factor $1 - \alpha(1 - \phi(\gamma))$. 

However, rather than viewing finding an equilibrium of~\eqref{eq:inn} as finding a zero of a non-Euclidean monotone operator, it is also possible to view it as an operator splitting problem. In particular, in the spirit of~\cite[Theorem~1]{EW-JZK:20}, we prove that finding a fixed point of $\Phi(Ax + Bu + b)$ corresponds to an appropriate operator splitting problem under suitable assumptions on $\Phi$. However, first we must define the proximal operator.
\begin{definition}[Proximal operator~{\cite[Definition~12.23]{HHB-PLC:17}}]
	Suppose $\map{f}{\real^n}{{]-\infty,\infty]}}$ is a proper lower semicontinuous convex function. Then the \emph{proximal operator of $f$} evaluated at $x \in \real^n$ is
	\begin{equation}
	\prox_f(x) = \argmin_{z \in \real^n} \frac{1}{2}\|x - z\|_{2}^2 + f(z).
	\end{equation}
\end{definition}
\begin{proposition}
	Suppose $\phi$ is the proximal operator of a continuously differentiable convex function $f$. Then finding an equilibrium point $x^*(u)$ of~\eqref{eq:inn} is equivalent to the operator splitting problem $(\OF+\OG)(x^*(u)) = 0$, where
	\begin{equation}\label{eq:splitting}
	\OF(z) = (I_n - A)(z) - (Bu + b), \qquad \OG(z) = df(z),
	\end{equation}
	where we denote $df(z) = (f'(z_1), \dots, f'(z_n))^\top$.
\end{proposition}
\begin{proof}
	First, we note that computing an equilibrium point of~\eqref{eq:inn} is equivalent to computing a fixed-point $x = \Phi(Ax + Bu + b)$. Since $\phi(x_i) = \prox_f(x_i)$, by~\cite[Proposition~16.44]{HHB-PLC:17} we have that $\Phi(x) = \OJ_{df}(x)$. Thus, the fixed-point problem is equivalent to
	\begin{equation}
	x = \OJ_{df}(Ax + Bu + b),
	\end{equation}
	which implies the result since this corresponds to Algorithm~\ref{alg:FwdBwd} with $\alpha = 1$. Therefore, any equilibrium point of~\eqref{eq:inn} is a zero of the splitting problem $(\OF + \OG)(x) = 0$ where $\OF$ and $\OG$ are defined as in~\eqref{eq:splitting}.
\end{proof}
We note that although this assumption on $\phi$ appears restrictive, if the assumption of smoothness is relaxed, many common activation functions satisfy the assumption, as is noted in the following little-known proposition.
\begin{proposition}[{\cite[Proposition~2.4]{PLC-JCP:08}}]
	Let $\map{\phi}{\real}{\real}$. Then $\phi$ is the proximal operator of a proper lower semicontinuous convex function $\map{f}{\real}{{]-\infty,\infty]}}$ if and only if $\phi$ satisfies
	$$0 \leq \frac{\phi(x) - \phi(y)}{x - y} \leq 1, \quad \text{for all } x,y \in \real, x\neq y.$$
\end{proposition}

For the LeakyReLU activation function, it is known that the $f$ corresponding to $\phi$ is given by $f(z_i) = \frac{1-a}{2a}\min\{z_i,0\}^2,$~\cite[Table~1]{JL-CF-ZL:19} which is continuously differentiable and which can be written in vector form $df(z) = \frac{1-a}{a} \min\{z,0\}$. Moreover, $df$ is Lipschitz with constant $(1-a)/a$. Now we will show that under the sufficient condition $\gamma < 1$, $\OF$ is strongly monotone with respect to the norm $\|\cdot\|_{\infty,[\eta]^{-1}}$ and $\OG$ is monotone with respect to the same norm. 

Since $\gamma < 1$,
$$-\mu_{\infty,[\eta]^{-1}}(-(I_n - A)) = 1 - \mu_{\infty,[\eta]^{-1}}(A) = 1 - \gamma > 0,$$
which implies $\OF$ is strongly monotone with monotonicity parameter $1 - \gamma$. Moreover, checking that $\OG$ is monotone is straightforward since $df$ is Lipschitz and $\jac{df}(z)$ is diagonal for every $z \in \real^n$ for which it exists and has diagonal entries in $[0,(1-a)/a]$. As a consequence, for almost every $z \in \real^n$, $\mu_{\infty,[\eta]^{-1}}(-\jac{df}(z)) \leq 0$, which implies monotonicity of $\OG$ with respect to $\|\cdot\|_{\infty,[\eta]^{-1}}$.

Therefore, we can consider different operator splitting algorithms to compute the equilibrium of~\eqref{eq:inn}. First, the forward-backward splitting method as applied to this problem is
\begin{equation}\label{eq:FB-RNN}
\begin{aligned}
x_{k+1} &= \OJ_{\alpha\OG}((1 - \alpha)x_k + \alpha (Ax_k + Bu + b)) \\
&= \prox_{\alpha f}((1 - \alpha)x_k + \alpha (Ax_k + Bu + b)).
\end{aligned}
\end{equation}
Since $\OF$ is Lipschitz, this iteration is guaranteed to converge to the unique fixed point of~\eqref{eq:inn}. Moreover, the contraction factor for this iteration is $1 - \alpha(1 - \gamma)$ for $\alpha \in {]0, \frac{1}{1 - \min_i (A)_{ii}}]}$, with contraction factor being maximized at $\alpha^* = \frac{1}{1 - \min_i (A)_{ii}}$. Note that compared to the iteration~\eqref{eq:averagediteration}, the forward-backward iteration has a larger allowable range of step sizes and improved contraction factor at the expense of computing a proximal operator at each iteration.


Alternatively, the fixed point may be computed by means of the Peaceman-Rachford splitting algorithm, which can be written 
\begin{equation}\label{eq:PR-RNN}
\begin{aligned}
x_{k+1/2} &= (I_n + \alpha(I_n - A))^{-1}(z_{k} + \alpha(Bu + b)), \\
z_{k+1/2} &= 2x_{k+1/2} - z_k, \\
x_{k+1} &= \prox_{\alpha f}(z_{k+1/2}), \\
z_{k+1} &= 2x_{k+1} - z_{k+1/2}. 
\end{aligned}
\end{equation}
Since both $\OF$ and $\OG$ are Lipschitz, this iteration converges to the unique fixed point of~\eqref{eq:inn}. Moreover, the contraction factor is $\ds\frac{1 - \alpha(1 - \gamma)}{1 + \alpha(1 - \gamma)}$ for 
$\ds\alpha \in {\Big]0, \min\Big\{\frac{1}{1 - \min_i (A)_{ii}}, \frac{a}{1-a}\Big\}\Big]},$
which comes from the Lipschitz constants of $\OF$ and $\OG$. In other words, the contraction factor is improved for Peaceman-Rachford compared to the forward-backward splitting, but the stepsize is additionally limited by the Lipschitz constant of $df$. For recurrent neural networks where $A$ has large negative diagonal entries and $(I_n + \alpha(I_n - A))$ may be easily inverted, this splitting method may be preferred.

\subsection{Numerical implementations}
\begin{figure}
	\includegraphics[width=0.95\linewidth]{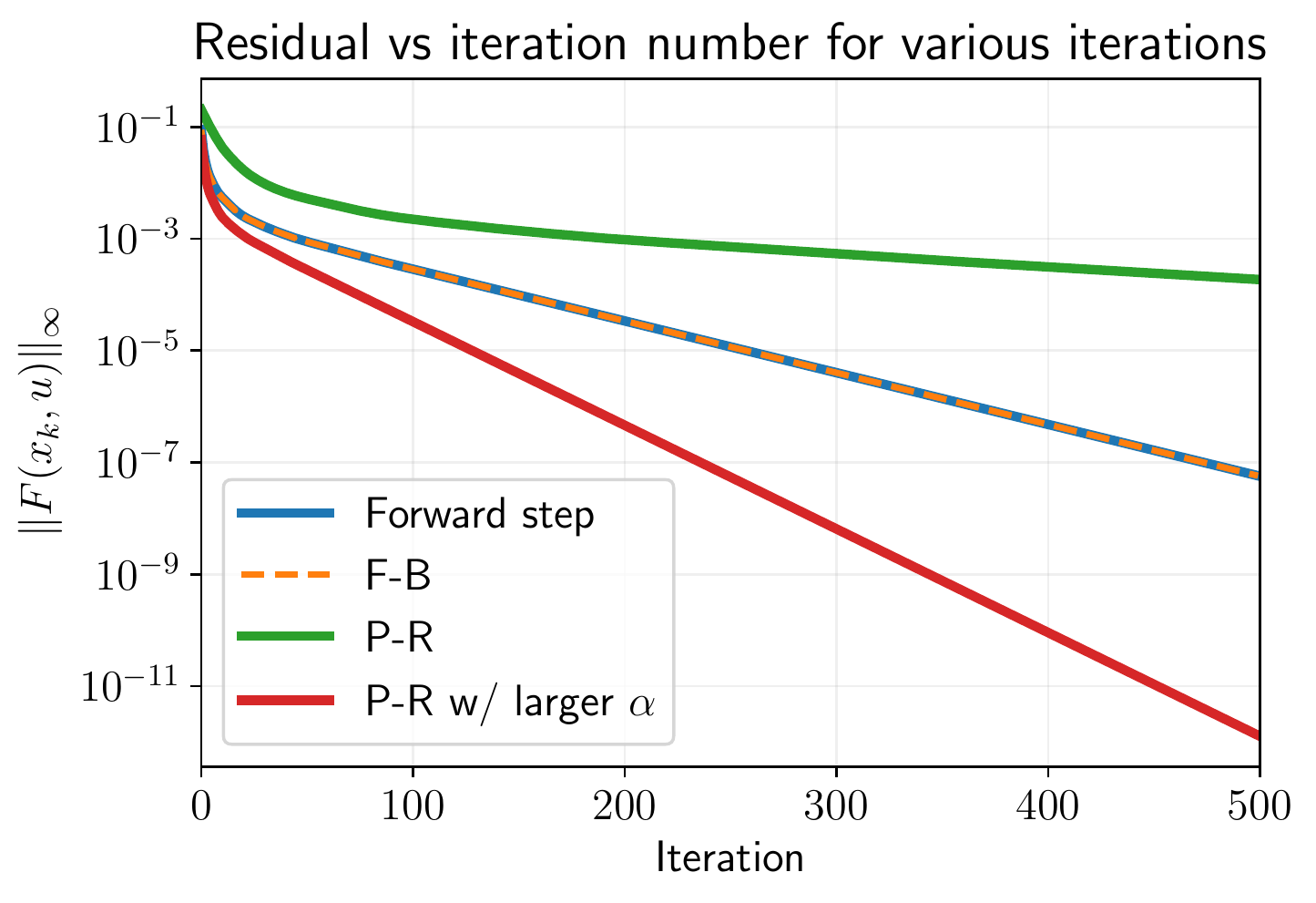}
	\caption{Residual versus number of iterations for forward-step method~\eqref{eq:averagediteration}, forward-backward (F-B) splitting~\eqref{eq:FB-RNN}, and Peaceman-Rachford (P-R) splitting~\eqref{eq:PR-RNN} for computing the equilibrium of the recurrent neural network~\eqref{eq:inn}. Curves for the forward-step method and forward-backward splitting are directly on top of one another.}\label{fig:iterations}
\end{figure}
To assess the efficacy of the iterations in~\eqref{eq:averagediteration},~\eqref{eq:FB-RNN}, and~\eqref{eq:PR-RNN}, we generated $A, B, b, u$ in~\eqref{eq:inn} and applied the iterations to compute the equilibrium. We generate $A \in \real^{200 \times 200}, B \in \real^{200 \times 50}, u \in \real^{50}, b \in \real^{200}$ with entries normally distributed as $A_{ij}, B_{ij}, b_{i} \sim \mathcal{N}(0, 1/\sqrt{200})$ and $u_{i} \sim \mathcal{N}(0, 1/\sqrt{50})$. To ensure that $A \in \real^{200 \times 200}$ satisfies the constraint $\mu_{\infty,[\eta]^{-1}}(A) < 1$ for some $\eta \in \realpositive^n$, we pick $[\eta] = I_n$ and project $A$ onto the convex polytope $\setdef{A \in \real^{n \times n}}{\mu_{\infty}(A) \leq 0.99}$. We additionally computed $\mu_2(A) \approx 1.0034$, so $\OF$ is not strongly monotone with respect to $\|\cdot\|_2$.

For all iterations, we initialize $x_0$ at the origin and for the Peaceman-Rachford iteration, we additionally initialize $z_0$ at the origin. We set $a = 0.1$ in LeakyReLU and for each iteration pick the largest theoretically allowable stepsize, which for the forward-step method and forward-backward splitting was $\ds \frac{1}{1-\min_i (A)_{ii}} \approx 0.9015$. For Peaceman-Rachford splitting, the largest theoretically allowable stepsize was $a/(1-a) \approx 0.1111$, but we also simulated using Peaceman-Rachford splitting with $\alpha = 0.9015$. The plots of the residual $\|x_k - \Phi(Ax_k + Bu + b)\|_{\infty} = \|F(x_k,u)\|_{\infty}$ versus the number of iterations is shown in Figure~\ref{fig:iterations}.

We see that, in this instance, both forward-step and forward-backward splitting methods for computing the equilibrium of~\eqref{eq:inn} converge at the same rate. This result agrees with the theory since $\gamma = 0.99 > 0$, so that $\phi(\gamma) = \gamma$ and the estimated contraction factor for both the forward step method and forward-backward splitting is $1 - \alpha(1 - \gamma) \approx 0.9910$. For the Peaceman-Rachford splitting method, for the theoretically largest allowable $\alpha = 1/9$, the estimated contraction factor is $\frac{1-\alpha(1-\gamma)}{1+\alpha(1-\gamma)} \approx 0.9978$, which is very close to $1$ and thus justifies the slow rate of convergence for the iterations in this case. However, if we let $\alpha = 0.9015$ as in the other methods, we observe a significant acceleration in the convergence of these iterations. Increasing the range of allowable stepsizes and the tightness of the Lipschitz constants to be more consistent with the empirical results remains an interesting topic of future research.
\section{Conclusion}
We develop a non-Euclidean monotone operator framework with an emphasis on operators which are monotone with respect to finite-dimensional $\ell_1$ and $\ell_\infty$ norms. Many classical algorithms for computing zeros of monotone operators including the forward step method, proximal point method, and splitting methods such as forward-backward splitting and Peaceman-Rachford splitting are directly applicable in our framework and can exhibit improved convergence rates compared to their corresponding algorithms in Euclidean spaces. We apply our results to recurrent neural network equilibrium computation and empirically demonstrate that applying splitting methods yields improved rates of convergence to the equilibria as compared to other methods.

Topics of future research include (i) tightening the Lipschitz estimates of the operator splitting techniques, (ii) extending the results to include infinite-dimensional Banach spaces and set-valued operators $\OF$, and (iii) applying this framework for robustness analysis of control systems and machine learning models.




\bibliographystyle{plainurl+isbn}
\bibliography{alias,Main,FB}

\end{document}